\documentclass[11pt, reqno]{amsart}
\usepackage{amsmath, amsthm, amscd, amsfonts, amssymb, mathrsfs, graphicx, color}
\usepackage[bookmarksnumbered, colorlinks, plainpages]{hyperref}
\hypersetup{colorlinks=true,linkcolor=red, anchorcolor=green, citecolor=cyan, urlcolor=red, filecolor=magenta, pdftoolbar=true}

\theoremstyle{plain}
\newtheorem{theorem}{Theorem}[section]
\newtheorem{corollary}[theorem]{Corollary}
\newtheorem{lemma}[theorem]{Lemma}
\newtheorem{proposition}[theorem]{Proposition}

\theoremstyle{definition}

\theoremstyle{remark}
\newtheorem{remark}{Remark}[section]
\newtheorem{example}{Example}[section]

\numberwithin{equation}{section}

\numberwithin{table}{section}

\numberwithin{figure}{section} 
\begin{document}
\title[Weighted Hermite-Hadamard Inequality]
{A new variant of weighted Hermite-Hadamard Inequalities and applications}
\author[M. Ra\"{\i}ssouli, L. Tarik, M. Chergui]{Mustapha Ra\"{\i}ssouli$^{1}$, Lahcen Tarik$^{2}$ and Mohamed Chergui$^3$}
\address{$^{1}$ Department of Mathematics, Science Faculty, Moulay Ismail University, Meknes, Morocco.}
\address{$^{2}$ LAGA-Laboratory, Science Faculty,  Ibn Tofail University, Kenitra, Morocco.}
\address{$^{3}$ Department of Mathematics, CRMEF-RSK, EREAM Team, LaREAMI-Lab, Kenitra, Morocco.}

\email{\textcolor[rgb]{0.00,0.00,0.84}{raissouli.mustapha@gmail.com}}
\email{\textcolor[rgb]{0.00,0.00,0.84}{lahcen.tarik@uit.ac.ma}}
\email{\textcolor[rgb]{0.00,0.00,0.84}{chergui\_m@yahoo.fr}}


\begin{abstract}
In this article, we focus on establishing a new variant of Hermite-Hadamard type inequalities for operator convex maps using an appropriate probability measure. To underline the usefulness of these inequalities, we investigate some refinements of some well-known operator inequalities, as well as the definition of new weighted operator means.
\end{abstract}

\maketitle
	
\section{Introduction}

Let $I$ be a nonempty interval in $\mathbb{R}$. A function $f:I\rightarrow{\mathbb R}$ is said convex (resp. concave) if, for any $a,b\in I$ and any $\lambda\in[0,1]$, the following inequality holds:

\begin{equation}\label{CF}
f\big((1-\lambda)a+\lambda b\big)\leq (\geq) (1-\lambda)f(a)+\lambda f(b).
\end{equation}
For convex function $f$, the following double inequality
\begin{equation}\label{HHI}
f\Big(\dfrac{a+b}{2}\Big)\leq\int_0^1f\Big((1-t)a+tb\Big)dt\leq\dfrac{f(a)+f(b)}{2},
\end{equation}
is satisfied for any $a,b\in I$. If $f$ is concave, then \eqref{HHI} is reversed.

These inequalities are known in the literature as Hermite-Hadamard inequalities (HHI). They play a very important role in mathematics and other disciplines. For example, they are investigated to get some useful inequalities and refine others.

This motivates us to undertake this research, which aims to explore possible versions of these inequalities for the case of operators that may allow significant applications. To this end, we need to recall certain concepts and results.

In this paper, we consider a complex Hilbert space $(H,\langle,\rangle)$, and we denote the $C^*$-algebra of all bounded linear operators defined on  $H$ by $\mathcal{B}(H)$. An operator $A \in \mathcal{B}(H)$ is called positive ($A\geq0$), if it is self-adjoint and $\langle Ax,x\rangle$ for every $x\in H$. $\mathcal{B}^+(H)$ will stand for the subset of $\mathcal{B}(H)$ consisting of all positive operators, and we refer to the set of positive invertible operators as $\mathcal{B}^{+*}(H)$. A partial order is defined on $\mathcal{B}^+(H)$ by stating for any $A, B$ what follows
\begin{gather*}
B\ge A \text{ if and only if } B-A\in \mathcal{B}^+(H).
\end{gather*}

Now, let us consider a real-valued function $f:I\rightarrow{\mathbb R}$. We denote by $\mathcal{S}_I(H)$ the class of all self-adjoint operators with spectra in $I$. The function $f$ is said operator monotone if $A \le B$ implies $f(A) \le f(B)$, where $A,B\in\mathcal{S}_I(H)$, and $f(A)$ is defined using the usual functional calculus techniques.

The operator version of \eqref{CF} can be stated as follows. The function $f$ is said operator convex (resp. concave) on $I$  if the following inequality
\begin{equation}\label{OCF}
f\big((1-\lambda)A+\lambda B\big)\le (\ge) (1-\lambda)f(A)+\lambda f(B),
\end{equation}
holds for any $A, B\in\mathcal{S}_I(H)$ and any $\lambda \in [0,1]$.

If $f$ is an operator convex function on $I$, the following Hermite-Hadamard inequalities for operator case
\begin{gather}\label{HHOI}
f\left(\frac{A+B}{2}\right)\leq \int_{0}^{1} f\big((1-t) A+tB\big)\;dt \leq \frac{f(A)+f(B)}{2},
 \end{gather}
hold for any $A, B \in \mathcal{S}_I(H)$. If $f$ is operator concave on $I$, the inequalities \eqref{HHOI} are reversed. Inequalities \eqref{HHOI} generalize and refine classical operator inequalities. For further details, the interested reader can consult, for instance, \cite{BT,D2,D3,D4, RCT} and the related references cited therein.

A significant result for convex operator functions is the extended integral Jensen inequality, see \cite{Hansen et al(2007)} for instance. In this context, for a convex function $f: I \rightarrow \mathbb{R}$ and a bounded Radon measure $\nu$, the following Jensen operator inequality
\begin{gather}\label{HIneq}
f\left(\int_I\, Q_td\nu(t)\right)\le\int_I f(Q_t)d\nu(t),
\end{gather}
holds, for bounded field $\left(Q_t\right)_{t \in I}\subset \mathcal{S}_I(H)$ with norm continuous on $I$.

Furthermore, if $f$ is a $C^1$-class operator convex function on $I$, the following inequalities \cite{Dragomir(2021)}
\begin{gather}\label{DIneq}
D f(A)(B-A) \leq f(B)-f(A) \leq D f(B)(B-A),
\end{gather}
hold for any $A,B \in \mathcal{S}_I(H)$, where $D f(A)(B)$ represents the directional derivative of $f$ at $A$ in the direction $B$ defined by
\begin{gather*}
D f(A)(B):= \lim_{t\to 0}\dfrac{f(A+t\,B)- f(A)}{t}.
\end{gather*}

For suitable choices for the function $f$, the inequalities \eqref{HHOI} generate many important operator mean inequalities. In fact, the applicability of inequalities \eqref{HHOI} is possible for operator means by the following characterization of means. Given an operator mean $\sigma$, there exists a unique positive operator monotone function $f_\sigma$ defined on the interval $(0,\infty)$ such as \cite{kubo and ando}:
\begin{equation}
A\sigma B=A^{1/2}f_\sigma\left(A^{-1/2}BA^{-1/2}\right)A^{1/2},
\end{equation}
with $f_\sigma(1)=1$.

The mean $\sigma$ is said to be symmetric if $A\sigma B= B\sigma A$ for any $A, B \in \mathcal{B}^{+*}(H)$, which is equivalent to $f_\sigma(x)=xf_\sigma\left(x^{-1}\right)$ for all $x>0$.
If the representative function $f_\sigma$ is differentiable at $1$ with the condition $f_\sigma^\prime(1)= \lambda$,  $\sigma$ is called $\lambda$-weighted mean \cite{Ud}.

The standard weighted means are the weighted arithmetic, harmonic and geometric operator means defined, respectively, as follows:
 $$A\nabla_\lambda B=(1-\lambda)A+\lambda B,\,A!_\lambda B=\Big((1-\lambda)A^{-1}+\lambda B^{-1}\Big)^{-1},\, A\sharp_\lambda B=A^{1/2}\Big(A^{-1/2}BA^{-1/2}\Big)^\lambda A^{1/2}.$$
These means satisfy the following inequalities:
\begin{equation}\label{WOMI1}
A!_\lambda B\leq A\sharp_\lambda B\leq A\nabla_\lambda B,
\end{equation}
and they are symmetric if and only if $\lambda=1/2$, in which case they are simply denoted as $A\nabla B$, $A!B$, and $A\sharp B$, respectively.

In the current paper, we will also use the logarithmic operator mean \cite{ALR,Raissouli(2009),Raissouli and Furuichi(2020)}, defined by:
\begin{equation}\label{LM}
L(A,B):=\left(\int_0^1A^{-1}!_tB^{-1}\,dt\right)^{-1}=\int_0^1A\sharp_tB\,dt,
\end{equation}
A weighted version of the logarithmic operator mean, denoted by $L_\lambda(A,B)$, can be defined for any $\lambda \in(0,1)$ by its representative function $f_\lambda$, defined as follows:
\begin{equation*}
f_\lambda(x)=\frac{1}{\log x}\left(\frac{1-\lambda}{\lambda}\left(x^\lambda-1\right)+\frac{\lambda}{1-\lambda} x^\lambda\left(x^{1-\lambda}-1\right)\right),
\end{equation*}
for all $x\in(0,\infty),\; x\neq1$, with $f_\lambda(1)=1$. The following operator mean inequalities were established in \cite{Pal et al(2016)}:
\begin{equation*}
S!_\lambda\,T\le S\sharp_\lambda T\le L_\lambda(S,T)\le S\nabla_\lambda T.
\end{equation*}

The rest of this manuscript is organized as follows: Section 2 is devoted to establishing new weighted Hermite-Hadamard inequalities, followed by Section 3, where we introduce new logarithmic operator means.

\section{New weighted Hermite-Hadamard inequalities}\label{sect2}

For the sake of simplicity, we extend the weighted arithmetic operator mean from $\mathcal{B}^{*+}(H)$ to $\mathcal{B}(H)$ by setting
$$A\nabla_\lambda B:=(1-\lambda)A+\lambda B,\,\text{ for any }\lambda\in[0,1].$$

Our first main result reads as follows.

\begin{theorem}\label{T1}
Let $f:I\rightarrow{\mathbb R}$ be an operator convex function. For any $\lambda\in(0,1)$ and $A,B\in \mathcal{S}_I(H)$, we have the following inequalities
\begin{equation}\label{WHHOI}
f\big(A\nabla_\lambda B\big)\leq\int_0^1f\big(A\nabla_t B\big)d\eta_\lambda(t)\leq f(A)\nabla_\lambda\,f(B),
\end{equation}
where $\eta_\lambda$ is the probability measure defined on $[0,1]$ by
\begin{equation}\label{RM}
d\eta_\lambda(t)=\frac{\lambda}{1-\lambda}t^\frac{2\lambda-1}{1-\lambda}dt.
\end{equation}
If $f:I\rightarrow{\mathbb R}$ is operator concave then \eqref{WHHOI} are reversed.
\end{theorem}
\begin{proof}
Let us notice at first that,
$$A\nabla_\lambda B=\int_0^1 A\nabla_t B\;d\eta_\lambda(t).$$
Applying Jensen operator integral inequality \eqref{HIneq}, with $Q_t=A\nabla_t B$ and $I=[0,1]$, we get
\begin{gather}
f(A\nabla_\lambda B)=f\left(\int_0^1 A\nabla_t Bd\eta_\lambda(t)\right)\leq\int_0^1 f\left(A\nabla_t B\right)d\eta_\lambda(t).
\end{gather}
In another part, using \eqref{OCF}, we obtain
\begin{gather}
\int_0^1f\big(A\nabla_t B\big)d\eta_\lambda(t)
\leq\int_0^1 f(A)\nabla_t\,f(B)d\eta_\lambda(t)=f(A)\nabla_\lambda\,f(B).
\end{gather}
This completes the proof.
\end{proof}

Inequalities \eqref{WHHOI} will be called the weighted Hermite-Hadamard operator inequalities, \textit{w-HHOI} in short. It is worth mentioning that \textit{w-HHOI} generalize \eqref{HHOI}. Indeed, if we take $\lambda=1/2$ and $d\eta_{1/2}(t)=dt$ thus \eqref{WHHOI} coincides with \eqref{HHOI}.

A straightforward application of \eqref{WHHOI} is stated in the following proposition.

\begin{proposition}
For any $x\in(0,\infty)$ and $y\in [2,\infty)$, we have
\begin{gather}\label{BetaFunction}
\dfrac{1}{x(1+x)^{y-1}}\le B(x,y)\le \dfrac{1}{x(1+x)},
\end{gather}
where $B(x,y)$ refers to the standard beta function, namely $B(x,y)=\int_0^1 t^{x-1}(1-t)^{y-1}dt$ for $x,y>0$.
If $y\in[1,2]$ then \eqref{BetaFunction} are reversed.
\end{proposition}
\begin{proof}
Let $x\in(0,\infty)$ and $y\in [2,\infty)$. The real valued function defined on $[0,1]$ by $f(t)=(1-t)^{y-1}$ is convex.
So, noticing that $\frac{x}{1+x}\in (0,1)$ and applying \eqref{WHHOI} in its scalar version, we get
$$\left(0\nabla_\frac{x}{1+x}1\right)^{y-1}\leq\int_0^1\left(0\nabla_t1\right)^{y-1}d\mu_\frac{x}{1+x}(t)\leq1\nabla_\frac{x}{1+x}0,$$
or, equivalently,
$$\frac{1}{(1+x)^{y-1}}\leq x\int_0^1\left(1-t\right)^{y-1}t^{x-1}dt\leq \frac{1}{1+x}.$$
Thus, the inequalities \eqref{BetaFunction} are proved. The second assertion in the proposition comes from the fact that $f$ is concave if $y\in[1,2]$.
\end{proof}

The \textit{w-HHOI} enable us to get a refinement for the standard Hermite-Hadamard operator inequality \eqref{HHOI} as pointed out in the following corollary.

\begin{corollary}
Let $f:I\rightarrow{\mathbb R}$ be operator convex. For any $\lambda\in(0,1)$ and $A,B\in \mathcal{S}_I(H)$, the following inequalities hold
\begin{equation}\label{RHHOI}
f\big(A\nabla B\big)\leq \int_0^1f\big(A\nabla_tB\big)dt\leq \int_0^1\int_0^1f\big(A\nabla_tB\big)\;d\eta_\lambda(t)d\lambda\leq f(A)\nabla f(B).
\end{equation}
\end{corollary}
If $f$ is operator concave then \eqref{RHHOI} are reversed.
\begin{proof}
Integrating \eqref{WHHOI} with respect to $\lambda\in(0,1)$ and using the left inequality in \eqref{HHOI}, we get \eqref{RHHOI}.
\end{proof}

Let $(\lambda, \alpha) \in(0,1)\times [0,1]$, and let $A, B\in\mathcal{S}_I(H)$. We define the following probability measure
\begin{equation}\label{eta}
d\sigma_{\lambda,\alpha} = (1-\alpha)d\eta_{\lambda} + \alpha d\eta_{1-\lambda},
\end{equation}
and, we set
\begin{equation}\label{M-eta}
{\mathcal M}_{\lambda,\alpha}(f;A,B)=\int_0^1 f\left(A\nabla_t B\right)d\sigma_{\lambda,\alpha}(t).
\end{equation}
Noticing that $d\sigma_{1-\lambda,1-\alpha}=d\sigma_{\lambda,\alpha}$, we deduce that for any $A, B\in \mathcal{S}_I(H)$ and $\lambda\in(0,1),\alpha\in[0,1]$, the following equality holds
\begin{equation}\label{M-eta2}
{\mathcal M}_{1-\lambda,1-\alpha}(f;A,B)={\mathcal M}_{\lambda,\alpha}(f;A,B).
\end{equation}

Otherwise, for $\nu,\alpha\in[0,1]$ we define the bivariate function
\begin{equation}\label{31a}
r(\nu,\alpha):=\nu+\alpha-2\nu\alpha.
\end{equation}

The basic properties of $(\nu,\alpha)\longmapsto r(\nu,\alpha)$, whose the proof is straightforward, are embodied in the following proposition.

\begin{proposition}
For any $\nu,\alpha\in[0,1]$ there hold:\\
(i) $r(\nu,\alpha)=\alpha\sharp_\nu(1-\alpha)=\nu\sharp_\alpha(1-\nu)$.\\
(ii) $0\leq r(\nu,\alpha)\leq 1$ and $r(\nu,\alpha)=r(\alpha,\nu)$.\\
(iii) $r(\nu,1/2)=r(1/2,\alpha)=1/2$.\\
(iv) $r(1-\nu,\alpha)=r(\nu,1-\alpha)=1-r(\nu,\alpha)$.\\
(v) $r(1-\nu,1-\alpha)=r(\nu,\alpha)$.
\end{proposition}

Now, we may state the following result.

\begin{theorem}
Let $f:I\rightarrow{\mathbb R}$ be an operator convex function. For any $\lambda\in(0,1),\alpha \in[0,1]$ and $A,B\in \mathcal{S}_I(H)$, the following inequalities hold
\begin{multline}\label{WHHOI2}
f\big(A\nabla_{r(\lambda,\alpha)}B\big)\le\int_0^1f(A\nabla_{r(\lambda,t)} B)\;d\eta_\alpha(t)\le f\big(A\nabla_\lambda\,B\big)\nabla_\alpha f\big(A\nabla_{1-\lambda}\,B\big)\\
\leq{\mathcal M}_{\lambda,\alpha}(f;A,B)\leq f(A)\nabla_{r(\lambda,\alpha)}\,f(B),
\end{multline}
where ${\mathcal M}_{\lambda,\alpha}(f;A,B)$ is defined by \eqref{M-eta} and $r(\lambda,\alpha)$ is defined by \eqref{31a}.

If $f$ is operator concave then \eqref{WHHOI2} are reversed.
\end{theorem}
\begin{proof}
By considering the measure $d\eta_\alpha$, and replacing in \eqref{WHHOI} $A$ and $B$ respectively with $A\nabla_\lambda B$ and $A\nabla_{1-\lambda}B$, we get
\begin{gather}\label{HHOI-alpha}
f\big(A\nabla_{r(\lambda,\alpha)}\,B\big)\leq\int_0^1f\big(A\nabla_{r(\lambda,t)}\,B\big)d\eta_\alpha(t)\leq f(A\nabla_\lambda B)\nabla_\alpha\,f(A\nabla_{1-\lambda}B).
\end{gather}
By applying \eqref{WHHOI} for $\lambda$ and for $1-\lambda$, then adding the obtained inequalities term by term, we find
\begin{gather}\label{HHOI-eta0}
f\big(A\nabla_\lambda\,B\big)\nabla_\alpha f\big(A\nabla_{1-\lambda}\,B\big)\leq\int_0^1f\big(A\nabla_t\,B\big)d\sigma_{\lambda,\alpha}(t)\leq f(A)\nabla_{r(\lambda,\alpha)}\,f(B).
\end{gather}
The desired results follow by combining \eqref{HHOI-alpha} and \eqref{HHOI-eta0}.
\end{proof}

From the previous theorem we deduce the following corollary which implies that the mapping $\lambda \mapsto {\mathcal M}_{\lambda,\alpha}(f;A,B)$, for fixed $\alpha\in[0,1]$, can be extended over the whole interval $[0,1]$.

\begin{corollary}
Let $f:I\rightarrow{\mathbb R}$ be operator convex (resp. operator concave). For any $A,B\in \mathcal{S}_I(H)$, we have
$$\lim_{\lambda\to 0}{\mathcal M}_{\lambda,\alpha}(f;A,B)=f(A)\nabla_\alpha f(B)\;\; \text{and}\;\; \lim_{\lambda\to 1}{\mathcal M}_{\lambda,\alpha}(f;A,B)=f(A)\nabla_{1-\alpha} f(B).$$
\end{corollary}
\begin{proof}
Since $f:I\rightarrow{\mathbb R}$ is operator convex on $I$, it is norm-continuous on $I$. Thus, the desired results follow from \eqref{WHHOI2}.
\end{proof}

\begin{theorem}\label{T-eta}
Let $f:I\rightarrow{\mathbb R}$ be operator convex, $A, B\in \mathcal{S}_I(H)$ and $\lambda\in[0,1]$. Then there hold
\begin{multline}\label{LWHHOI-eta}
f(A\nabla_\lambda B)\le\int_0^1 f\big((A\nabla_\lambda\,B)\nabla\,(A\nabla_x\,B)\big)\, d\eta_\lambda(x)
\le \int_0^1\mathcal{M}_{\lambda,\frac{1}{2}}(f;A\nabla_\lambda\,B,A\nabla_x\,B)\,d\eta_\lambda(x)\\
\le
 f\big(A\nabla_\lambda B\big)\nabla\,\mathcal{M}_{\lambda,0}(f;A,B)
 \le \int_0^1f\big(A\nabla_x\,B\big)\, d\eta_\lambda(x).
\end{multline}
If $f$ is operator concave then \eqref{LWHHOI-eta} are reversed.
\end{theorem}
\begin{proof}
For any $A, B\in \mathcal{S}_I(H)$ and $\lambda\in[0,1]$, we have
\begin{align*}
f\left(A \nabla_\lambda B\right) & =f\left(\left(A \nabla_\lambda B\right) \nabla\left(A \nabla_\lambda B\right)\right) \\
& =f\left(\left(A\nabla_\lambda B\right) \nabla\Big(\int_0^1 A\nabla_x B\, d \eta_\lambda(x)\Big)\right) \\
& =f\left(\int_0^1\left(A\nabla_\lambda B\right) \nabla\left(A\nabla_x B\right)\, d \eta_\lambda(x)\right) \\
& \leq \int_0^1 f\Big(\left(A\nabla_\lambda B\right) \nabla\left(A\nabla_x B\right)\Big)\, d \eta_\lambda(x) \quad \text { by convexity of } f . \\
& \leq \int_0^1\left\{\int_0^1 f\left(\left(A\nabla_\lambda B\right) \nabla_t\left(A\nabla_x B\right)\right)\, d \sigma_{\lambda,\frac{1}{2}}(t)\right\}\, d \eta_\lambda(x) \quad \text { by \eqref{HHOI-eta0}, with } \alpha=\frac{1}{2} \\
& =\int_0^1 \mathcal{M}_{\lambda,\frac{1}{2}}\left(f ; A\nabla_\lambda B, A\nabla_x B\right)\, d \eta_\lambda(x) \\
& \leq \int_0^1\big[f\left(A\nabla_\lambda B\right) \nabla f\left(A\nabla_x B\right)\big]\, d \eta_\lambda(x) \quad \text { by \eqref{HHOI-eta0} } \\
& =f\left(A\nabla_\lambda B\right) \nabla \int_0^1 f\left(A\nabla_x B\right)\, d \eta_\lambda(x) \\
& \leq \int_0^1 f\left(A\nabla_x B\right)\, d \eta_\lambda(x) \quad \text { by \eqref{WHHOI}}.
\end{align*}
Thus, the proof is completed.
\end{proof}

The following lemma, which provides a refinement and a reverse of \eqref{OCF}, will be needed in the sequel. See \cite{D1,RAI0} for instance.

\begin{lemma}\label{L1}
Let $f:I\rightarrow{\mathbb R}$ be operator convex. Then the following inequalities
\begin{multline}\label{ROCF}
m(s,t)\Big(f(A)\nabla_sf(B)-f\big(A\nabla_sB\big)\Big)
\leq f(A)\nabla_tf(B)-f\big(A\nabla_tB\big)\\
\leq M(s,t)\Big(f(A)\nabla_sf(B)-f\big(A\nabla_sB\big)\Big),
\end{multline}
hold for any $A,B\in \mathcal{S}_I(H)$ and $s,t\in(0,1)$, where we set
\begin{equation}\label{mM}
m(s,t):=\min\left(\frac{t}{s},\frac{1-t}{1-s}\right),\;\; M(s,t):=\max\left(\frac{t}{s},\frac{1-t}{1-s}\right).
\end{equation}
If $f$ is operator concave then \eqref{ROCF} are reversed.
\end{lemma}

We also need the following lemma.

\begin{lemma}\label{L2}
For any $s,t,\lambda\in(0,1)$, the following relations hold
\begin{equation}\label{m}
\int_0^1m(s,t)\;d\eta_\lambda(t)=\dfrac{1-\lambda}{1-s}\left(1-s^\frac{\lambda}{1-\lambda}\right).
\end{equation}
\begin{equation}\label{M}
\int_0^1M(s,t)\;d\eta_\lambda(t)=(1-s)^{-1}\nabla_\lambda s^{-1}-\dfrac{1-\lambda}{1-s}\left(1-s^\frac{\lambda}{1-\lambda}\right),
\end{equation}
\end{lemma}
\begin{proof}
Notice that, $0\leq t\leq s$ if and only if $\frac{t}{s}\leq\frac{1-t}{1-s}$. Then, we can write
$$\int_0^1m(s,t)\;d\eta_\lambda(t)=\frac{1}{s}\int_0^std\eta_\lambda(t)+\frac{1}{1-s}\int_s^1(1-t)d\eta_\lambda(t).$$
The use of \eqref{RM} leads to
$$td\eta_\lambda(t)=\dfrac{\lambda}{1-\lambda}t^\frac{\lambda}{1-\lambda} \,\text{ and } (1-t)d\eta_\lambda(t)=\dfrac{\lambda}{1-\lambda}\left(t^\frac{2\lambda-1}{1-\lambda}-t^\frac{\lambda}{1-\lambda}\right),$$
and,
$$\int_0^1\big(M(s,t)+m(s,t)\big)d\eta_\lambda(t)=(1-s)^{-1}\nabla_\lambda s^{-1}.$$
An integration over $t$ gives us the results. The details are straightforward and therefore omitted here.
\end{proof}

The following result provides a refinement and a reverse of the right inequality in \eqref{WHHOI}.

\begin{theorem}\label{T2}
Let $f:I\rightarrow{\mathbb R}$ be operator convex. For any $s,\lambda\in[0,1]$ and $A,B\in \mathcal{S}_I(H)$ the following inequalities hold
\begin{multline}\label{RWHHOIR}
\alpha(s,\lambda)\Big(f(A)\nabla_sf(B)-f\big(A\nabla_sB\big)\Big)
\leq f(A)\nabla_\lambda f(B)-\int_0^1f\big(A\nabla_tB\big)d\eta_\lambda(t)\\
\leq \mu(s,\lambda)\Big(f(A)\nabla_sf(B)-f\big(A\nabla_sB\big)\Big),
\end{multline}
where we set
$$\alpha(s,\lambda):=\dfrac{1-\lambda}{1-s}\left(1-s^\frac{\lambda}{1-\lambda}\right),\,\text{ and } \mu(s,\lambda):=(1-s)^{-1}\nabla_\lambda s^{-1}-\dfrac{1-\lambda}{1-s}\left(1-s^\frac{\lambda}{1-\lambda}\right).$$
If $f$ is operator concave then \eqref{RWHHOIR} are reversed.
\end{theorem}
\begin{proof}
Multiplying all sides of \eqref{ROCF} by $d\eta_\lambda(t)$ and then integrating with respect to $t\in[0,1]$, we obtain the desired inequalities by the use of \eqref{m} and \eqref{M}. The details are simple and therefore omitted here for the reader.
\end{proof}
Taking $s=\lambda$ in Theorem \ref{T2}, we get the following corollary.

\begin{corollary}
For $f:J\rightarrow{\mathbb R}$ operator convex, $\lambda\in[0,1]$ and $S,T\in \mathcal{C}_J(H)$ there hold
\begin{multline}\label{a}
\left(1-\lambda^\frac{\lambda}{1-\lambda}\right)\Big(f(S)\nabla_\lambda f(T)-f\big(S\nabla_\lambda T\big)\Big)
\leq f(S)\nabla_\lambda f(T)-\int_0^1f\big(S\nabla_tT\big)d\nu_\lambda(t) \\
\leq \left(1+\lambda^\frac{\lambda}{1-\lambda}\right)\Big(f(S)\nabla_\lambda f(T)-f\big(S\nabla_\lambda T\big)\Big).
\end{multline}
If $f$ is operator concave, the inequalities \eqref{a} are reversed.
\end{corollary}

In the following result, we investigate a refinement and a reverse of the left inequality in \eqref{WHHOI}.

\begin{theorem}\label{T3}
Let $f:I\longrightarrow\mathbb{R}$ be an operator convex function  of class $C^1(I)$. For any $A, B\in \mathcal{S}_I(H)$ and $\lambda\in[0,1]$, we have
\begin{equation}\label{RWHHOIL}
0\le \int_0^1f(A\nabla_t\,B)d\eta_\lambda(t)-f\big(A\nabla_\lambda B\big)\le\int_0^1(t-\lambda)Df(A\nabla_tB)(B-A)d\eta_\lambda(t).
\end{equation}
\end{theorem}
\begin{proof}
Using the fact that $f$ is an operator convex function of $C^1(I)$, we can apply \eqref{DIneq} for $U=A\nabla_\lambda B$ and $V= A\nabla_t B$. It yields,
$$Df(U)(V-U)\le f(V)-f(U)\le Df(V)(V-U),$$
and therefore,
\begin{equation}\label{DIneq3}
(t-\lambda)Df\big(A\nabla_\lambda B\big)(B-A) \le f\big(A\nabla_t B\big)-f\big(A\nabla_\lambda B\big)
\le (t-\lambda)Df\big(A\nabla_t B\big)(B-A).
\end{equation}
Multiplying all sides of \eqref{DIneq3} by $d\eta_\lambda(t)$, and then integrating with respect to $t\in[0,1]$, we get \eqref{RWHHOIL}.
\end{proof}

The version of inequalities \eqref{RWHHOIL} for the scalar case is stated in the following result.

\begin{corollary}
Let $f: I\rightarrow{\mathbb R}$ be a convex function, $a,b\in I$ and $\lambda\in[0,1]$. We have
\begin{equation}\label{RWHHIR}
0\le \int_0^1f(a\nabla_t b)d\eta_\lambda(t)-f\left(a\nabla_\lambda b\right) 
\le (b-a)\int_0^1(t-\lambda)f^\prime(a\nabla_t b)d\eta_\lambda(t).
\end{equation}
\end{corollary}
\begin{proof}
Considering $H=\mathbb{R}$ in Theorem \ref{T3}, we get
$$0\le \int_0^1f(a\nabla_t b)d\eta_\lambda(t)-f\left(a\nabla_\lambda b\right)\\
\le (b-a)\int_0^1(t-\lambda)f^\prime(a\nabla_t b)d\eta_\lambda(t),$$
and by the use of the left inequality of \eqref{WHHOI}, we get \eqref{RWHHIR}
\end{proof}

\section{Application to new weighted logarithmic operator means} \label{sect3}

Based on our previous findings, we introduce some new weighted operator means and establish some of their properties. We begin by stating the following lemma, useful for our subsequent discussions.

\begin{lemma}\label{L3}
Let $\lambda\in[0,1]$. The two following functions
\begin{gather}\label{31}
x\longmapsto f_\lambda(x)=\left(\int_0^1\big(1\nabla_t x\big)^{-1}d\eta_\lambda(t)\right)^{-1},
\end{gather}
\begin{gather}\label{32}
x\longmapsto g_\lambda(x)=\int_0^11\sharp_t x\,d\eta_\lambda(t)
\end{gather}
are both operator monotone on $(0,+\infty)$. Furthermore, we have
\begin{gather}\label{33}
f_\lambda(1)=g_\lambda(1)=1\, \text{ and } \dfrac{df_\lambda}{dx}(1)=\dfrac{dg_\lambda}{dx}(1)=\lambda.
\end{gather}
\end{lemma}
\begin{proof}
We consider the function $p:x\mapsto x^{-1}$ and $q:x\mapsto \int_0^1p\big(1\nabla_tx\big)d\eta_\lambda(t)$. Then $p$ and $q$ are operator monotone
decreasing on $(0,+\infty)$. So, $f_\lambda=poq$ is operator monotone on $(0,+\infty)$.

On the other part, $x\mapsto x^t$ is operator monotone for $t\in[0,1]$, and then $g_\lambda$ is as well. 

The first two relations in \eqref{33} are straightforward. We will prove the last two ones.
For fixed $\lambda\in[0,1]$, we set $u(x):=\big(f_\lambda(x)\big)^{-1}$ for $x>0$. Applying the inequalities \eqref{HHI} for the convex function $x\mapsto 1/x$ on $(0,\infty)$, we obtain
$$\big(1-\lambda+\lambda x\big)^{-1}\leq u(x)\leq 1-\lambda+\lambda x^{-1},$$
for any $x>0$.
Thus, noticing that $u(1)=1$, we have for any $x>1$ (resp. $x<1$) the following inequalities
\begin{equation}\label{u}
\frac{-\lambda}{1-\lambda+\lambda x}\leq\,(\ge)\frac{u(x)-u(1)}{x-1}\leq\,(\ge)-\frac{\lambda}{x},
\end{equation}
Passing to the limit when $x$ tends to 1, we obtain $u^{\prime}(1)=-\lambda$. Hence, $\frac{df_\lambda}{dx}(1)=\lambda$.

On another part, by applying the inequalities \eqref{HHI} again for $x\mapsto x^t$, with $t\in[0,1]$, we get
$$x^\lambda\leq g_\lambda(x)\leq 1-\lambda+\lambda x.$$
Noticing that $g_\lambda(1)=1$, and using similar techniques as used for the computation of $\frac{df_\lambda}{dx}(1)$, we get $\frac{dg_\lambda}{dx}(1)=\lambda$.
\end{proof}

In what follows, for any $A, B\in\mathcal{B}(H)^{+*}$ and $\lambda\in[0,1]$, we set
\begin{equation}\label{WRL1}
\mathbf {L}_\lambda(A,B):=\left(\int_0^1\big(A\nabla_tB\big)^{-1}d\eta_\lambda(t)\right)^{-1}
\end{equation}
and
\begin{equation}\label{WRL2}
\mathbb{L}_\lambda(A,B):=\int_0^1A\sharp_t Bd\eta_\lambda(t).
\end{equation}

\begin{proposition}
 ${\mathbf L}_\lambda$ and $\mathbb{L}_\lambda$ are $\lambda$-weighted operator means. Moreover, we have
\begin{gather}\label{34}
\mathbf{L}_{1/2}(A,B)=\mathbb{L}_{1/2}(A,B)=L(A,B).
\end{gather}
Thus, ${\mathbf L}_\lambda$ and $\mathbb{L}_\lambda$ will be called $\lambda$-weighted logarithmic operator means.
\end{proposition}
\begin{proof}
Noticing that
\begin{gather}
\mathbf {L}_\lambda(A,B)=A^{\frac{1}{2}}f_\lambda\left(A^{-\frac{1}{2}}BA^{-\frac{1}{2}}\right)A^{\frac{1}{2}}\,
\text{ and }\,
\mathbb {L}_\lambda(A,B)=A^{\frac{1}{2}}g_\lambda\left(A^{-\frac{1}{2}}BA^{-\frac{1}{2}}\right)A^{\frac{1}{2}}.
\end{gather}
By virtue of Lemma \ref{L3}, we deduce that $\mathbf{L}_\lambda$ and $\mathbb{L}_\lambda$ are $\lambda$-weighted operator means whose the representing functions are $f_\lambda $ and $ g_\lambda$, respectively.
\end{proof}

\begin{remark}
(i) Equalities \eqref{34} justify calling both $\mathbf{L}_\lambda(A,B)$ and $\mathbb {L}_\lambda(A,B)$ weighted logarithmic operator means.\\
(ii) $L_\lambda(A,B)$, $\mathbf{L}_\lambda(A,B)$ and $\mathbb {L}_\lambda(A,B)$ are mutually distinct, as shown in the following example.
\end{remark}

\begin{example}
Let us consider $\lambda=3/4$, $A=\begin{pmatrix}
1 &0\\ 0& 2\end{pmatrix}$ and $B=\begin{pmatrix}
2 &0\\ 0& 1\end{pmatrix}$. Employing straightforward real integration techniques, we obtain the following results:
\begin{gather*}
{\mathbf L}_{3/4}(A,B)\approx \begin{pmatrix}
1.6964 &0\\ 0& 1.2004\end{pmatrix},\,
{\mathbb L}_{3/4}(A,B)\approx \begin{pmatrix}
1.7258 & 0\\0 & 1.2228\end{pmatrix} \text{ and } \\
L_{3/4}(A,B)\approx \begin{pmatrix}
1.7051 & 0\\0 & 1.2088 \end{pmatrix}.
\end{gather*}
\end{example}

The following result deals with the natural question of comparison between some of the previous weighted operator means.

\begin{proposition}\label{Estimate}
For any $A, B\in\mathcal{B}^{+*}(H)$ and $\lambda\in[0,1]$ there hold
\begin{gather}\label{NWOMI1}
A!_\lambda B \leq \mathbf {L}_\lambda(A,B)\leq A\nabla_\lambda B,
\end{gather}
\begin{gather}\label{NWOMI2}
 A\sharp_\lambda B\le \mathbb{L}_\lambda(A,B)\le A\nabla_\lambda B.
\end{gather}
\end{proposition}
\begin{proof}
By applying \eqref{WHHOI} for the operator convex function $x\mapsto 1/x$ on $(0,+\infty)$, we get
$$\left(A\nabla_\lambda B\right)^{-1}\le \int_0^1\left(A\nabla_t B\right)^{-1}d\eta_\lambda(t)\le A^{-1}\nabla_\lambda B^{-1},$$
hence \eqref{NWOMI1}. 
To show \eqref{NWOMI2}, we apply \eqref{RWHHOIR} to the convex function $t\mapsto x^t,\,x\in(0,\infty)$ on $[0,1]$, and we get
$$x^\lambda\le \int_0^1x^td\eta_\lambda(t)\le 1-\lambda+\lambda x.$$
By the techniques of functional calculus, we can substitute $x$ by $A^{-\frac{1}{2}}BA^{-\frac{1}{2}}$ in this latter inequality for getting
$$\left(A^{-\frac{1}{2}}BA^{-\frac{1}{2}}\right)^\lambda\le \int_0^1\left(A^{-\frac{1}{2}}BA^{-\frac{1}{2}}\right)^td\eta_\lambda(t)\le 1-\lambda+\lambda \left(A^{-\frac{1}{2}}BA^{-\frac{1}{2}}\right).$$
This leads to \eqref{NWOMI2} after multiplying all sides, at left and at right, by $A^{\frac{1}{2}}$.
\end{proof}

In the rest of this section we propose some estimations and refinements of the left inequality in \eqref{NWOMI1} and the right inequality in \eqref{NWOMI2} proved in Theorem \ref{Estimate}. We have the following result.

\begin{theorem}
For any $s,\lambda\in[0,1]$ and $A,B\in \mathcal{B}^{+*}(H)$, we have
\begin{multline}\label{ENWOMI1}
\alpha(s,\lambda)\left((A!_s B)^{-1}-(A\nabla_s B)^{-1}\right)\le (A!_\lambda B)^{-1}-\mathbf{L}_\lambda^{-1}(A,B)\\
\le \mu(s,\lambda)\left((A!_s B)^{-1}-(A\nabla_s B)^{-1}\right),
\end{multline}
and
\begin{gather}\label{ENWOMI2}
\alpha(s,\lambda)\big(A\nabla_s B-A\sharp_s B\big)\le A\nabla_\lambda B-\mathbb{L}_\lambda(A,B)\le \mu(s,\lambda)\left(A\nabla_s B-A\sharp_s B\right).
\end{gather}
\end{theorem}
\begin{proof}
Inequalities \eqref{ENWOMI1} are deduced by applying \eqref{RWHHOIR} to the operator convex function $f(x)=\frac{1}{x}$, on $(0,+\infty)$.
For \eqref{ENWOMI2}, we apply \eqref{RWHHOIR} to the convex function $t\mapsto x^t,\,x\in(0,\infty)$ on $[0,1]$, we get
$$
\alpha(s,\lambda)\big(1-s+sx-x^s\big)\le 1-\lambda+\lambda x-\int_0^1x^td\eta_\lambda(t)\le \mu(s,\lambda)\big(1-s+sx-x^s\big),
$$
Substituting $x$ by $A^{-\frac{1}{2}}BA^{-\frac{1}{2}}$ and multiplying, at left and at right, all sides by $A^{\frac{1}{2}}$, we get \eqref{ENWOMI2}.
\end{proof}

\begin{corollary}
Let $\lambda\in[0,1]$ and $A, B\in \mathcal{B}^{+*}(H)$. The following inequalities hold
\begin{gather}\label{RNWOMI1}
(A!_\lambda B)^{-1}\nabla_{\mu(\lambda,\lambda)}\,(A\nabla_\lambda B)^{-1}\le \mathbf{L}^{-1}_\lambda(A,B)
\le (A!_\lambda B)^{-1}\nabla_{\alpha(\lambda,\lambda)}\,(A\nabla_\lambda B)^{-1},
\end{gather}
and
\begin{gather}\label{RNWOMI2}
(A\nabla_\lambda B)\nabla_{\mu(\lambda,\lambda)}\,(A\sharp_\lambda B)\le \mathbb{L}_\lambda(A,B)
\le (A\nabla_\lambda B)\nabla_{\alpha(\lambda,\lambda)}\,(A\sharp_\lambda B).
\end{gather}
\end{corollary}
\begin{proof}
Taking $s=\lambda$ in \eqref{ENWOMI1} and \eqref{ENWOMI2}, we find respectively \eqref{RNWOMI1} and \eqref{RNWOMI2}.
\end{proof}

The following remark may be of interest for the reader.

\begin{remark}
(i) Using the right inequality of \eqref{RNWOMI1}, we get the following refinement of the left inequality in \eqref{NWOMI1}
$$A!_\lambda\,B\le (A!_\lambda\,B)!_{\alpha(\lambda,\lambda)}\,(A\nabla_\lambda\,B)\le\mathbf{L}_\lambda(A,B).$$
(ii) The right inequality in \eqref{RNWOMI2} refines the right inequality in \eqref{NWOMI2}. Indeed, we have
$$\mathbb{L}_\lambda(A,B)
\le (A\nabla_\lambda B)\nabla_{\alpha(\lambda,\lambda)}\,(A\sharp_\lambda B)\le A\nabla_\lambda B.$$
\end{remark}

\begin{theorem}
For any $A, B\in\mathcal{B}^{+*}(H)$ and $\lambda\in[0,1]$, we have
\begin{multline}\label{RNWOMI3}
\mathbf{L}_\lambda(A,B)\le (A\nabla_\lambda\,B)!\,\mathbf{L}_\lambda(A,B)\le \left[\int_0^1\mathcal{M}_{\lambda,\frac{1}{2}}(A\nabla_\lambda\,B,A\nabla_x\,B)d\eta_\lambda(x)\right]^{-1}\\
\le \left[\int_0^1\Big((A\nabla_\frac{\lambda+x}{2} B\Big)^{-1}d\eta_\lambda(x)\right]^{-1}\le A\nabla_\lambda\,B.
\end{multline}
\end{theorem}
\begin{proof}
We apply \eqref{LWHHOI-eta} for the operator convex function $f(x)=\dfrac{1}{x}$ on $(0,+\infty)$.
\end{proof}

To state another result involving the new operator mean $\mathbf{L}_\lambda(A,B)$, we need the following lemma which gives
the derivative of an operator function, see \cite{Dragomir(2021)} for instance.

\begin{lemma}\label{L4}
Let $X, Y\in\mathcal{B}^{*+}(H)$. The directional derivative of the operator function $t\longmapsto f(t):=t^{-1},\,t\in (0,\infty)$ at the point $X$ in the direction $Y$ is given by
$$D f (Y )(X)=-Y^{-1}XY^{-1}.$$
\end{lemma}

\begin{theorem}
Let $A, B\in \mathcal{B}^{*+}(H)$ and $\lambda\in[0,1]$. Then there holds:
\begin{equation}
0\le \mathbf{L}_\lambda^{-1}(A,B)-(A\nabla_\lambda B)^{-1}
\le\int_0^1(\lambda-t)(A\nabla_t B)^{-1}(B-A)(A\nabla_t B)^{-1}d\eta_\lambda(t).
\end{equation}
\end{theorem}
\begin{proof}
We first apply \eqref{RWHHOIL} with the operator convex function $t\mapsto t^{-1}$ on $(0,\infty)$. Using Lemma \ref{L4} and formula \eqref{WRL1}, we conclude the proof.
\end{proof}

We now point out some inequalities involving some mixed operator means.

\begin{theorem}
Let $A,B\in\mathcal{S}_{(0,\infty)}(H)$ and $s,\lambda\in[0,1]$. The following inequalities hold
\begin{gather}\label{Mixte1}
A\nabla_\lambda\,\big(A\sharp_s\,B\big)\le\int_0^1A\sharp_s\,(A\nabla_t\,B)d\eta_\lambda(t)\le A\sharp_s\,(A\nabla_\lambda\,B).
\end{gather}
\end{theorem}
\begin{proof}
We apply \eqref{WHHOI} for the operator concave function $f(x)=x^s$ on $(0,+\infty)$ for $s\in[0,1]$. It yields,
$$I\nabla_\lambda\big(A^{-\frac{1}{2}}BA^{-\frac{1}{2}}\big)^s
\le\int_0^1\left(I\nabla_t\big(A^{-\frac{1}{2}}BA^{-\frac{1}{2}}\big)\right)^sd\eta_\lambda(t)\le\left(I\nabla_\lambda
\big(A^{-\frac{1}{2}}BA^{-\frac{1}{2}}\big)\right)^s,$$
for any $A,B\in\mathcal{S}_{(0,\infty)}(H)$. Multiplying at left and at right both sides by $A^{1/2}$, we get \eqref{Mixte1}.
\end{proof}

\begin{corollary}
Let $A,B\in\mathcal{S}_{(0,\infty)}(H)$ and $\lambda\in[0,1]$ the following inequalities hold
\begin{gather}\label{Mixte2}
A\nabla_\lambda\,\mathbb{L}_\lambda(A,B)\le \int_0^1\mathbb{L}(A,A\nabla_t\,B)d\eta_\lambda(t)\le \mathbb{L}_\lambda(A,A\nabla_\lambda\,B).
\end{gather}
\end{corollary}
\begin{proof}
Multiplying all sides of \eqref{Mixte1} by $d\eta_\lambda(s)$ and then integrating with respect to $s\in[0,1]$, we obtain the desired inequalities.
\end{proof}

\begin{theorem}
Let $A,B\in\mathcal{S}_{(0,\infty)}(H)$ and $(r,s,\lambda)\in[0,1]^3$. The following inequalities hold
\begin{multline}\label{Mixte3}
\alpha(s,\lambda)\Big(A\nabla_s\,\big(A\sharp_r\,B\big)-A\sharp_r\big(A\nabla_s\,B\big)\Big)
\le A\nabla_\lambda\,\big(A\sharp_r\,B\big)-\int_0^1A\sharp_r\,(A\nabla_t\,B)d\eta_\lambda(t)\\
\le \beta(s,\lambda)\Big(A\nabla_s\,\big(A\sharp_r\,B\big)-A\sharp_r\big(A\nabla_s\,B\big)\Big).
\end{multline}
\end{theorem}
\begin{proof}
Let us consider $f(x):=x^r$ with $r\in[0,1]$. Employing the inequalities \eqref{RWHHOIR} for the operator concave function $f$ on $(0,+\infty)$ for $I$ and $C=A^{-1/2}BA^{-1/2}$, we get
$$
\alpha(s,\lambda)\left[I\nabla_sC^r-\left(I\nabla_sC\right)^r\right]
\leq I\nabla_\lambda C^r-\int_0^1\left(I\nabla_tC\right)^rd\eta_\lambda(t)
\leq \beta(s,\lambda)\left[I\nabla_sC^r-\left(I\nabla_sC\right)^r\right].
$$
Multiplying at left and at right both sides by $A^{1/2}$, we deduce \eqref{Mixte3}.
\end{proof}

\begin{remark}
If we take $s=\lambda$ in \eqref{Mixte3}, we find the following reverses of the inequalities \eqref{Mixte1}
\begin{multline}\label{Mixte4}
\alpha(\lambda,\lambda)\Big(A\nabla_\lambda\,\big(A\sharp_r\,B\big)-A\sharp_r\big(A\nabla_\lambda\,B\big)\Big)
\le A\nabla_\lambda\,\big(A\sharp_r\,B\big)-\int_0^1A\sharp_r\,(A\nabla_t\,B)d\eta_\lambda(t)\\
\le \beta(\lambda,\lambda)\Big(A\nabla_\lambda\,\big(A\sharp_r\,B\big)-A\sharp_r\big(A\nabla_\lambda\,B\big)\Big).
\end{multline}
\end{remark}

Finally, we have the following corollary.

\begin{corollary}
Let $A, B\in\mathcal{S}_{(0,\infty)}(H)$ and $s,\lambda\in[0,1]$. We have the following inequalities
\begin{multline}\label{Mixte5}
\alpha(s,\lambda)\Big(A\nabla_s\,\mathbb{L}_\lambda(A,B)-\mathbb{L}_\lambda(A,A\nabla_\lambda\,B)\Big)
\le A\nabla_\lambda\,\mathbb{L}_\lambda(A,B)-\int_0^1\mathbb{L}_\lambda(A,A\nabla_tB)d\nu_\lambda(t)\\
\le \beta(s,\lambda)\Big(A\nabla_s\,\mathbb{L}_\lambda(A,B)-\mathbb{L}_\lambda(A,A\nabla_\lambda B)\Big).
\end{multline}
\end{corollary}
\begin{proof}
Multiplying at left and at right all sides of \eqref{Mixte3} by $d\nu_\lambda(r)$ and then integrating with respect to $r\in[0,1]$, we obtain the desired inequalities.
\end{proof}

\end{document}